\documentclass[article]{elsarticle}
\usepackage{graphicx}
\usepackage{adjustbox}
\usepackage{pgf,tikz}
\usepackage{tkz-tab,tkz-fct}
\usetikzlibrary{shapes,shapes.callouts,shapes.geometric}
\usetikzlibrary{arrows.meta}
\usetikzlibrary{automata}
\usepackage{booktabs} 
\usepackage{tikz}
\usetikzlibrary{arrows}
\usetikzlibrary {shapes.geometric}
\usepackage[]{algorithm2e}
\usetikzlibrary{positioning}
\usetikzlibrary{shapes.geometric}
\usepackage{oplotsymbl}
\usepackage{bigints}
\usepackage{lineno,hyperref}
\usepackage{caption}
\usepackage{amsfonts,amsmath,amssymb}
\usepackage{array,multirow,makecell}
\modulolinenumbers[5]
\newtheorem{theorem}{Theorem}

\newdefinition{defi}{Definition}
\newproof{proof}{Proof}

\newtheorem{lemma}{Lemma}
\begin{document}
\begin{frontmatter}
\title{Analytical solutions of virus propagation model in blockchain networks}
\author[*]{Youness Chatibi}
\address[*]{Mohammed V University in Rabat, Faculty of Sciences, Department of Mathematics,\\ Center CeReMAR, Laboratory LMSA, P.O. 1014 RP, Rabat, Morocco.}
\begin{abstract}
The main goal of this paper is to find analytical solutions of a system of nonlinear ordinary differential equations arising in the virus propagation in blockchain networks. The presented method reduces the problem to an Abel differential equation of the first kind and solve it directly. 
\end{abstract}
\end{frontmatter}
\begin{section}{Introduction}
In the last decade, many scientific problems have been shown as mathematical model with systems of nonlinear ordinary differential equations, notably in epidemiology \cite{MOREIRA1997,NUCCI2004,MUSHANYU2016,YOSHIDA2023,OYARZUN1994,LUO1995,ARCIERO2004,NAGY2004,AWAWDEH2009,ARENAS2009,NEKORKIN2011}, physics \cite{JONES1985,CHAKRAVARTY2003,NIKOLOV2021,WHITE1980,WINTERNITZ1984,BOUNTIS1986,OLMO1987,SACHDEV1996,KAZANTSEV1997,BERTOTTI2008,FILO2012,ALBERT2012}, chemistry \cite{COHEN1975,BOURNE2014,KORLIE2007} and computer science \cite{REN2012,ZHU2012,HANSO1992}. However, finding analytical solutions for these systems is often very difficult. Few nonlinear systems can be solved explicitly, except that some numerical approximations of the solutions had been proposed \cite{RAFEI2007, LARA2008,GOH2010}. \\
Mathematical and computer modeling have received considerable attention due to their practical importance in controlling and predicting viral spread. It is very significant to study changes in infected host populations in computer networks. Note that the spread of malicious code is in many ways similar to the spread of biological viruses. In 2008, a new technology called blockchain technology was proposed by Satoshi Nakamoto \cite{NAKAMOTO2008}. It is based on a distributed ledger that records transactions in blocks without the need for a central authority of trust. Each block contains a set of transactions and has a hash link to the previous block. If transactions occur at the same time, they are recorded in the same block \cite{ZHANG2019}. Therefore, blockchain modeling is very important because malware virus classes are rapidly evolving and the cybersecurity risks of cryptocurrency networks such as bitcoin are increasing \cite{PAPADIMITRIOU2020}.\\
This work presents analytical solutions of a model in a blockchain network that appears the interactions between viruses, protected and unprotected computers.
\end{section}
\section{Mathematical description of the problem}
We consider $x_1$, $x_2$ and $x_3$ as population levels of viruses, protected and unprotected systems. The mathematical formulation of the model is based on the diagram given by Fig. \ref{vvp}, which yield the following system of nonlinear ordinary differential equations \cite{BHARDWAJ2020}:
 \begin{equation}
\begin{cases}
   \dfrac{dx_1}{dt}=-d_1 x_1+b_2 x_2  &  \medskip\\ 
   \dfrac{dx_2}{dt}=b_1 x_1-d_2 x_2-k_2 x_1 x_3 & \medskip\\
   \dfrac{dx_3}{dt}=-d_3 x_3+k_1 x_1 x_2,\label{system}
\end{cases}    
 \end{equation}
subject to the non-negative initial conditions $x_{1}(t_0)=x_{1}^0, x_{2}(t_0)=x_{2}^0, x_{3}(t_0)=x_{3}^0$, where $t_0\geq0$ is a real number. The parameters in the system (\ref{system}) are strictly positive and listed in the following table:\\
\begin{center}
\begin{tabular*}{350pt}{@{\extracolsep\fill}ll@{\extracolsep\fill}}
\toprule
\textbf{Parameter} & \textbf{Description}  \\
\midrule
\textbf{$d_1$}  & Coefficient of decay of virus\\ 
\textbf{$d_2$} & Coefficient of decay of vulnerable systems\\
\textbf{$d_3$} & Coefficient of decay of protected systems\\
\textbf{$b_1$} & Coefficient of susceptibility of vulnerable system\\
\textbf{$b_2$} 
& Coefficient of availability of vulnerable system\\
\textbf{$k_1$} & Coefficient of encounter between virus and protected systems\\
\textbf{$k_2$} & Coefficient of interaction between virus and vulnerable system\\
\bottomrule
\end{tabular*}
    \captionof{table}{Description of the model parameters}    
    \label{table1}
\end{center}
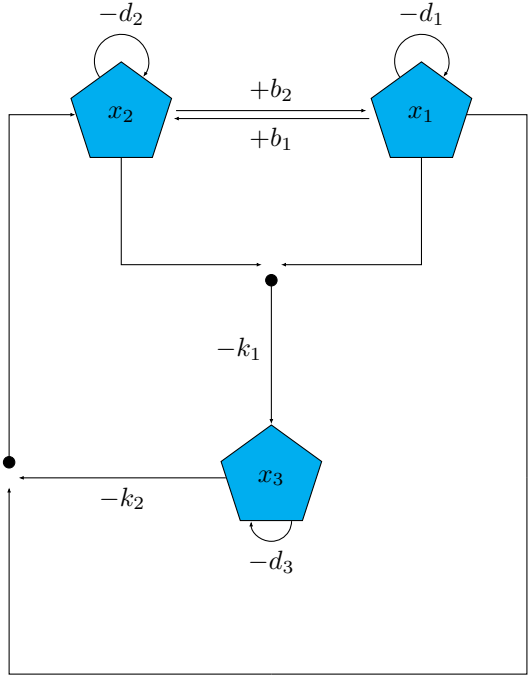
\begin{figure}[!ht]
\centering
\newcommand{\arccpp}[5][]{	
	\begin{scope}
	\coordinate (centre) at (#2);\coordinate (debut) at (#3);\coordinate (fin) at (#4);
	\pgfmathanglebetweenpoints{\pgfpointanchor{centre}{center}}{\pgfpointanchor{debut}{center}};
		\FPeval{\angd}{\pgfmathresult}
	\pgfmathanglebetweenpoints{\pgfpointanchor{centre}{center}}{\pgfpointanchor{fin}{center}};
		\FPeval{\ang}{\pgfmathresult}
	\pgfmathparse{greater(\angd,\ang)}
		\ifthenelse{\equal{\pgfmathresult}{1}} {\FPeval{\angf}{360+\ang}} {\FPeval{\angf}{\ang}};
	\path [draw,#1] let \p1 =($(centre)-(debut)$) in (debut) arc (\angd:\angf:{veclen(\x1,\y1)}) #5;
	\end{scope}}
	
\tikzset{auto,node distance=2cm,
	block/.style={draw, fill=cyan, regular polygon, regular polygon sides=5, minimum height=4em, minimum width=2em},
	flechedroite/.style={->,shorten <=3pt,shorten >=3pt},
	>={Latex[length=2pt]}}
\begin{tikzpicture}[scale=0.1]
	\coordinate (A) at (0,0);
	\node (sum) {};
	\node (Ablock) [state,block,above left=of sum] {$x_2$};
	\node (Bblock) [state,block,above right=of sum] {$x_1$};
	\node (Cblock) [state,block,below=of sum] {$x_3$};
	\draw [flechedroite] (Ablock.5)--(Bblock.175) node [midway,above] {$+b_2$};
	\draw [flechedroite] (Bblock.-175)--(Ablock.-5) node [midway,below] {$+b_1$};
	\node (noeudCgauche) [left=of Cblock,xshift=-5ex] {};
	\node (noeudCdroite) [right=of Cblock,xshift=5ex,inner sep=0pt,outer sep=0pt] {};
	\node (noeudCbas) [below=of Cblock,inner sep=0pt,outer sep=0pt] {};
	\draw [->] (Cblock)--(noeudCgauche) node [below,midway] {$-k_2$};
	\draw [*->] (noeudCgauche)|-(Ablock);
	\draw [->] (noeudCbas.-90)-|(noeudCgauche);
	\draw (noeudCbas.-90)-|(noeudCdroite.0);
	\draw (noeudCdroite.0)|-(Bblock);
	\draw [->] (Ablock)|-(sum);
	\draw [->] (Bblock)|-(sum);
	\draw [*->] (sum)--(Cblock) node [midway,left] {$-k_1$};
	\arccpp [<-] {Ablock.90}{Ablock.60}{Ablock.120}{node [midway,above] {$-d_2$}}
	\arccpp [<-] {Bblock.90}{Bblock.60}{Bblock.120}{node [midway,above] {$-d_1$}}
	\arccpp [<-] {Cblock.-90}{Cblock.-115}{Cblock.-65}{node [midway,below] {$-d_3$}}
\end{tikzpicture}\caption{Representation of the problem}
\label{vvp}
\end{figure}
\section{Analytical resolution of the model}
\begin{theorem}
The system (\ref{system}) can be reduced to an ordinary differential equation (called Li\'enard equation) determined by
\begin{equation}
\frac{d^{2} x_{1}}{d t^2}+A\frac{d x_1}{d t}+B(x_{1})=0,\label{Lienard}
\end{equation}
where $A$ and $B$ are polynomial functions.
\end{theorem}
\begin{proof}
    In this model, the population is constant (i.e.) $x_{1}+x_{2}+x_{3}=N$. Then $\frac{d x_1}{d t}+\frac{d x_2}{d t}+\frac{d x_3}{dt}=0$, therefore, for all $t \geq 0$ the resolution of this system is based on the determination of the function $x_{1}=x_1(t)$ because after that 
\begin{align}
  x_2(t)&=\dfrac{1}{b_{2}}\bigg(\dfrac{dx_{1}}{dt}+d_{1}x_{1}\bigg),\label{eq1}\\
  x_3(t)&=N-x_1(t)-x_2(t).\label{eq2}
\end{align} 
Substitution of (\ref{eq2}) into (\ref{system}) gives us the following ordinary differential equations:
\begin{align}
\frac{d x_2}{d t}&=(b_1-k_{2}N)x_{1}+k_{2}x_{1}^{2}-d_{2}x_{2}+k_{2}x_{1}x_{2},\\
\dfrac{d x_3}{d t}&=d_{3}(x_{1}+x_{2})+k_{1}x_{1}x_{2}-d_{3}N.    
\end{align}
Therefore,
\begin{align*}
k_{2}\dfrac{d x_1}{d t}+(k_{1}+k_{2})\dfrac{d x_2}{d t}&=k_{1}\dfrac{dx_{2}}{dt}-k_{2}\dfrac{dx_{3}}{dt}\\
&=[(b_{1}-k_{2}N)k_{1}-d_{3}k_{2}]x_{1}-(d_{2}k_{1}+d_{3}k_{2})x_{2}+k_{1}k_{2}x_{1}^{2}+d_{3}k_{2}N.
\end{align*}
By differentiating Eq. (\ref{eq1}) and use it, we get
\begin{equation}
    a\frac{d^{2} x_{1}}{d t^2}+b\frac{d x_1}{d t}+cx_{1}^{2}+dx_{1}+e=0,
\end{equation}
where
\begin{align*}
    a&=\dfrac{k_{1}+k_{2}}{b_{2}},\quad b=k_{2}+\dfrac{k_{1}(d_{1}+d_{2})+k_{2}(d_{1}+d_{3})}{b_{2}},\quad c=-k_{1}k_{2},\\
    d&=-b_{1}k_{1}+k_{2}(d_{3}+k_{1}N)+\dfrac{d_{1}(d_{2}k_{1}+d_{3}k_{2})}{b_{2}},\quad \text{and}\quad e=-d_{3}k_{2}N.    
\end{align*}
If we denote $A=\frac{b}{a}$ and $B(x_{1})=\frac{cx_{1}^{2}+dx_{1}+e}{a}$, we obtain the Li\'enard ordinary differential equation \cite{LIENARD1928}
\begin{equation}
\frac{d^{2} x_{1}}{d t^2}+A\frac{d x_1}{d t}+B(x_{1})=0.
\end{equation}
\end{proof}
\begin{lemma}
The solutions of the Li\'enard equation (\ref{Lienard}) can be obtained by transforming it to an equivalent first kind first order Abel type equation given by 
\begin{equation}
\frac{d v}{d x_{1}}=Av^{2}+B(x_{1})v^{3}.\label{Abel}
\end{equation} 
\end{lemma}
\begin{proof}
By denoting $u=\frac{d x_1}{d t}$, Eq. (\ref{Lienard}) can be expressed as
\begin{equation}
u\frac{d u}{d x_{1}}+Au+B(x_{1})=0.\label{newLienard}   
\end{equation}
By introducing a new dependent variable $v=\frac{1}{u}$, Eq. (\ref{newLienard}) takes the form of an Abel differential equation of the first kind (\ref{Abel}).   
\end{proof}
\begin{theorem}
The solution of the Abel equation (\ref{Abel}) is given by
\begin{equation}
    v=\bigg(\dfrac{dx_{1}}{dt}\bigg)^{-1}=-(Ax_{1}+C)^{-1}\pm\big[Dx_{1}^{3}+Ex_{1}^{2}+Fx_{1}+G\big]^{-\frac{1}{2}},
\end{equation}
where $C,D,E,F$ and $G$ are constants.
\end{theorem}
\begin{proof}
For resolving Eq. (\ref{Abel}), we put $v=v_{1}+v_{2}$ such that
 \begin{align}
   \label{newsystem1}
   \dfrac{d v_1}{d x_{1}}&=Av_{1}^{2}, \\
   \label{newsystem2}
   \dfrac{dv_2}{d x_{1}}&=B(x_{1})v_{2}^{3}.
 \end{align}
For Eq. (\ref{newsystem1}), we have the solution
\begin{equation}
  v_{1}=-(Ax_{1}+C)^{-1}\quad\text{where}\quad C\quad\text{is constant.}  
\end{equation}
Otherwise, for (\ref{newsystem2}), we put $v_{2}=P(x_{1})^{r}$ where $P\in\mathbb{R}[x_{1}]^{*}$ and $r\in\mathbb{Q}^{*}$. We have 
$$\frac{d v_{2}}{d x_{1}}=rP'(x_{1})[v_{2}^{3}]^{s}$$
where $s=\frac{1}{3}\big(1-\frac{1}{r}\big)$. For $s=1$ and $rP'(x_{1})=B(x_{1})$, we get the solution 
\begin{equation}
    v_{2}=\big[-2\int B(x_{1})dx_{1}+C'\big]^{-\frac{1}{2}}=\big[Dx_{1}^{3}+Ex_{1}^{2}+Fx_{1}+G\big]^{-\frac{1}{2}},
\end{equation}
where $D=\frac{-2c}{3a}, E=\frac{-d}{a}, F=\frac{-2e}{a}$ and $G=-2C''+C'$ ($C'$ and $C''$ are integration's constants).\\\\
Note that $\Tilde{v}_{2}=-v_{2}$ is also a solution. Then Eq. (\ref{Abel}) admits two solutions that are
\begin{equation}
    v=\bigg(\dfrac{dx_{1}}{dt}\bigg)^{-1}=-(Ax_{1}+C)^{-1}\pm\big[Dx_{1}^{3}+Ex_{1}^{2}+Fx_{1}+G\big]^{-\frac{1}{2}}.\label{solutions}
\end{equation}
\end{proof}
\begin{theorem}
The general solutions $(x_1,x_2,x_3)$ of the system (\ref{system}) is given by   
\begin{equation}
x_{1}=\sum_{n=1}^{\infty}\rho_{n}^{\pm}t^{n},\quad x_2(t)=\dfrac{1}{b_{2}}\bigg(\dfrac{dx_{1}}{dt}+d_{1}x_{1}\bigg),\quad x_3(t)=N-x_1(t)-x_2(t),
\end{equation}
where the coefficients $\rho_{n}^{\pm}$ are defined by
\begin{align}
 \rho_{n}^{\pm}&=\dfrac{1}{n(\sigma_{1}^{\pm})^{ n}}\sum_{s_{1},s_{2},s_{3},\cdots}(-1)^{s_{1}+s_{2}+s_{3}+\cdots}\\
 &\quad\times\dfrac{n(n+1)\cdots(n-1+s_{1}+s_{2}+\cdots)}{s_{1}!s_{2}!s_{3}!\cdots}\bigg(\dfrac{\sigma_{2}^{\pm}}{\sigma_{1}^{\pm}}\bigg)^{s_{1}}\bigg(\dfrac{\sigma_{3}^{\pm}}{\sigma_{1}^{\pm}}\bigg)^{s_{2}}\cdots,\label{morse}
\end{align}
and the sum over $s$ values is restricted to partitions of $n-1$,
\begin{equation*}
s_{1}+2s_{2}+3s_{3}+\cdots=n-1.
\end{equation*}
\end{theorem}
\begin{proof}
By integrating (\ref{Abel}) we obtain
\begin{equation}
    t=-\frac{1}{A}\log{(Ax_{1}+C)}\pm\int\big[Dx_{1}^{3}+Ex_{1}^{2}+Fx_{1}+G\big]^{-\frac{1}{2}}.\label{inverseequation}
\end{equation}
Now, letting the equation
\begin{equation}
 P(x_{1})=0\quad\text{where}\quad P(x_{1})=Dx_{1}^{3}+Ex_{1}^{2}+Fx_{1}+G.\label{Tschirnhaus} 
\end{equation}
Note that, if we apply the Tschirnhaus method \cite{TSCHIRNHAUS1683} by putting $y_{1}=x_{1}+\frac{E}{3D}$,\\\\
Eq. (\ref{Tschirnhaus}) become
\begin{equation}
 \Tilde{P}(y_{1})=0\quad\text{where}\quad \Tilde{P}(y_{1})=y_{1}^{3}+Hy_{1}+I,\label{Tschirnhausmodified}
\end{equation}\\
where $H=\frac{3DF-E^{2}}{3D^{2}}$ and $I=\frac{2E^{3}-9DEF+27D^{2}G}{27D^{3}}$.\\\\
Eq. (\ref{Tschirnhausmodified}), can be solve by Cardano's method \cite{MUKUNDAN2010} which one of its solutions take the form
\begin{equation}
y_{1}^{(1)}=\bigg(\dfrac{-I-\sqrt{\Delta_{1}}}{2}\bigg)^{\frac{1}{3}}+\bigg(\dfrac{-I+\sqrt{\Delta_{1}}}{2}\bigg)^{\frac{1}{3}},
\end{equation}
where $\Delta_{1}=I^{2}+\dfrac{4H^{3}}{27}\geq0$. As $y_{1}-y_{1}^{(1)}$ divide $\Tilde{P}(y_{1})$ we have
\begin{equation}
\Tilde{P}(y_{1})=(y_{1}-y_{1}^{(1)})Q(y_{1}),
\end{equation}
where 
\begin{equation}
Q(y_{1})=y_{1}^{2}+(y_{1}+y_{1}^{(1)})y_{1}^{(1)}+H.
\end{equation}
Consequently, the two remaining roots of the equation $Q(y_{1})=0$ are
\begin{equation}
y_{1}^{(2)}=\dfrac{-y_{1}^{(1)}-\sqrt{\Delta_{2}}}{2},\qquad y_{1}^{(3)}=\dfrac{-y_{1}^{(1)}+\sqrt{\Delta_{2}}}{2},
\end{equation}
where $\Delta_{2}=y_{1}^{(1)2}-4(H+y_{1}^{(1)2})\geq0$. Whence,
\begin{equation}
    \Tilde{P}(y_{1})=(y_{1}-y_{1}^{(1)})(y_{1}-y_{1}^{(2)})(x_{1}-y_{1}^{(3)}).
\end{equation}
As $P(x_{1})=D\cdot\Tilde{P}(x_{1}+\frac{E}{3D})$ where $D=\frac{2b_{2}k_{1}k_{2}}{3(k_{1}+k_{2})}>0$, we have
\begin{equation}
    \int v_{2}dx_{1}=\int\dfrac{dx_{1}}{\sqrt{D(x_{1}+\theta_{1})(x_{1}+\theta_{2})(x_{1}+\theta_{3})}},
\end{equation}
with $\theta_{k}=\frac{E}{3D}-y_{1}^{(k)}\neq0$ for $k=1,2,3$.\\\\
By applying
\begin{equation}
    \dfrac{1}{\sqrt{x+\theta}}=\dfrac{1}{\sqrt{\theta}}\sum_{p=0}^{\infty}\binom{-\frac{1}{2}}{p}\bigg(\dfrac{x}{\theta}\bigg)^{p},
\end{equation}
where
\begin{equation}
    \binom{x}{m}=\dfrac{1}{m!}\prod_{i=0}^{m-1}(x-i)=\dfrac{\Gamma(x+1)}{\Gamma(m+1)\cdot\Gamma(x-m+1)},\quad\forall  m\in\mathbb{N},\quad\forall x\in\mathbb{R}
\end{equation}
is the generalized binomial coefficient, we get
\begin{equation}
    \dfrac{1}{\sqrt{D(x_{1}+\theta_{1})(x_{1}+\theta_{2})(x_{1}+\theta_{3})}}=\dfrac{1}{\sqrt{D}}\sum_{n=0}^{\infty}\mu_{n}(\theta_{1},\theta_{2},\theta_{3})x_{1}^{n},
\end{equation}
where 
\begin{equation}
    \mu_{n}(\theta_{1},\theta_{2},\theta_{3})=\displaystyle\sum_{p+q+r=n}\dfrac{\binom{-\frac{1}{2}}{p}\binom{-\frac{1}{2}}{q}\binom{-\frac{1}{2}}{r}}{\theta_{1}^{p+\frac{1}{2}}\cdot\theta_{2}^{q+\frac{1}{2}}\cdot\theta_{3}^{r+\frac{1}{2}}}.
\end{equation}
Consequently, the integration implies that
\begin{equation}
    \int v_{2}dx_{1}=\dfrac{1}{\sqrt{D}}\sum_{n=0}^{\infty}\dfrac{\mu_{n}(\theta_{1},\theta_{2},\theta_{3})}{n+1}x_{1}^{n+1}.
\end{equation}
Otherwise, since $\dfrac{1}{1+x}=\displaystyle\sum_{n=0}^{\infty}(-x)^{n}$ for $|x|<1$ and by integration we achieve:
\begin{equation}
    -\dfrac{1}{A}\log{(Ax_{1}+C)}=\sum_{n=0}^{\infty}\lambda_{n}(A,C)x_{1}^{n+1},
\end{equation}
where 
\begin{equation}
\lambda_{n}(A,C)=\dfrac{(-1)^{n+1}}{(n+1)C^{n+1}}A^{n}.
\end{equation} 
Then, Eq. (\ref{inverseequation}) become
\begin{equation}
t=\sum_{n=1}^{\infty}\lambda_{n-1}(A,C)x_{1}^{n}\pm\dfrac{1}{\sqrt{D}}\sum_{n=1}^{\infty}\dfrac{\mu_{n-1}(\theta_{1},\theta_{2},\theta_{3})}{n}x_{1}^{n}=\sum_{n=1}^{\infty}\sigma_{n}^{\pm}x_{1}^{n},\label{inverseserie}
\end{equation}
where 
\begin{equation}
 \sigma_{n}^{\pm}=\lambda_{n-1}(A,C)\pm\dfrac{1}{\sqrt{D}}\dfrac{\mu_{n-1}(\theta_{1},\theta_{2},\theta_{3})}{n}\quad\text{for}\quad n\ge1.   
\end{equation} 
Since, the series expansion of the inverse series is given by
\begin{equation}
x_{1}=\sum_{n=1}^{\infty}\rho_{n}^{\pm}t^{n},\label{x1}
\end{equation}
where the coefficients $\rho_{n}^{\pm}$ are defined by (for details see \cite{MORSE1953})
\begin{align}
 \rho_{n}^{\pm}&=\dfrac{1}{n(\sigma_{1}^{\pm})^{ n}}\sum_{s_{1},s_{2},s_{3},\cdots}(-1)^{s_{1}+s_{2}+s_{3}+\cdots}\\
 &\quad\times\dfrac{n(n+1)\cdots(n-1+s_{1}+s_{2}+\cdots)}{s_{1}!s_{2}!s_{3}!\cdots}\bigg(\dfrac{\sigma_{2}^{\pm}}{\sigma_{1}^{\pm}}\bigg)^{s_{1}}\bigg(\dfrac{\sigma_{3}^{\pm}}{\sigma_{1}^{\pm}}\bigg)^{s_{2}}\cdots,
\end{align}
and the sum over $s$ values is restricted to partitions of $n-1$,
\begin{equation*}
s_{1}+2s_{2}+3s_{3}+\cdots=n-1.
\end{equation*}
We note that if $0<\theta_{1}\theta_{2}\theta_{3}\neq\frac{C^{2}}{D}$, $\sigma_{1}^{\pm}\neq0$. The first few $\rho_{n}^{\pm}$ given by (\ref{morse}) are
\begin{align*}
 \rho_{1}^{\pm}&=\dfrac{1}{\sigma_{1}^{\pm}},\\
 \rho_{2}^{\pm}&=-\dfrac{1}{(\sigma_{1}^{\pm})^{3}}\sigma_{2}^{\pm},\\
 \rho_{3}^{\pm}&=\dfrac{1}{(\sigma_{1}^{\pm})^{5}}(2(\sigma_{2}^{\pm})^{2}-\sigma_{1}^{\pm}\sigma_{3}^{\pm}).
\end{align*}
Finally, we use (\ref{x1}), Eq. (\ref{eq1}) and Eq. (\ref{eq2}) to get $x_{2}$ and $x_{3}$.
\end{proof}
\section*{Conflicts of Interest}
The authors would declare that they have no known competing interests that could have appeared to influence this work.

\end{document}